\documentclass[10pt,leqno]{article}

\usepackage{amsmath,amssymb,amsthm,mathrsfs,dsfont,graphicx}

\usepackage[margin=3cm]{geometry} 

\usepackage{titlesec,hyperref}

\usepackage{color}

\usepackage{fancyhdr}
\usepackage{amsmath,amssymb} 
\usepackage{diagbox} 

\titleformat{\subsection}{\it}{\thesubsection.\enspace}{1pt}{}

\newtheorem{theo}{Theorem}[section]
\newtheorem{lemm}[theo]{Lemma}
\newtheorem{defi}[theo]{Definition}

\newtheorem{prop}[theo]{Proposition}
\newtheorem{rema}[theo]{Remark}
\numberwithin{equation}{section}

\allowdisplaybreaks 

\makeatletter
\tagsleft@false  
\makeatother

\begin{document}
\title{Blow-up phemomenon for the Geng-Xue system and related models
\hspace{-4mm}
}

\author{Song $\mbox{Liu}^1$ \footnote{Email: lius37@mail2.sysu.edu.cn} \quad and\quad
	Zhaoyang $\mbox{Yin}^1$\footnote{E-mail: mcsyzy@mail.sysu.edu.cn}\\
	$^1\mbox{School}$ of Science,\\ Shenzhen Campus of Sun Yat-sen University, Shenzhen, 518107, China}

\date{}
\maketitle
\hrule

\begin{abstract}
In this paper, we consider the Cauchy problem of the Geng-Xue system with cubic nonlinearity. Firstly, we prove a blow-up criteria in the besov space with low regularity. Secondly, we prove the blow-up phenomenon by using the method which does not require any conservation law. Finally, we extend our results to the b-family of two-component system with cubic nonlinearity.   \\
\vspace*{5pt}
\noindent {\it 2020 Mathematics Subject Classification}: 35B44, 35L05, 35G05.

\vspace*{5pt}
\noindent{\it Keywords}: Geng-Xue system, blow-up criteria, blow-up.
\end{abstract}

\vspace*{10pt}

\tableofcontents

\section{Introduction}
~~~~In this paper, we study the Cauchy problem for the Geng-Xue system with cubic nonlinearity \cite{geng2009extension} 
\begin{align}\label{(1.1)}
	\begin{cases}m_t+u\upsilon m_x+3\upsilon u_xm=0,\\n_t+u\upsilon n_x+3u\upsilon_xn=0,\\m=u-u_{xx},\quad n=v-v_{xx}.\end{cases}
\end{align}
The Geng-Xue system was first constructed by Geng and Xue using the zero-curvature equation method, which is a new CH-type integrable system. The authors established its Hamiltonian structure and proved that it also admits peakons. Himonas and Mantzavinos studied the well-posedness of the Geng-Xue equations in \cite{himonas2016novikov}, and they also pointed out that the data-to-solution map for the initial value problem of the Geng-Xue equation is not uniformly continuous in \(H^s \times H^s\) with \(s < \frac{3}{2}\).

For $v \equiv 1,$ the system \eqref{(1.1)} becomes the following well-known Degasperis-Procesi (DP) equation \cite{degasperis1999asymptotic} 
\begin{align*}
	m_t + u m_x + 3u_x m = 0, \quad m = u - u_{xx}.
\end{align*}
The DP equation is recognized as an alternative model for describing nonlinear shallow water dynamics \cite{constantin2010inverse,constantin2009hydrodynamical}. As established in \cite{degasperis2002new}, this equation possesses a bi-Hamiltonian structure and an infinite hierarchy of conservation laws, and it admits peakon solutions analogous to those of the Camassa-Holm (CH) equation. The CH equation \cite{camassa1993integrable,camassa1994new}, defined as 
\begin{align*}
	m_t + u m_x + 2u_x m = 0, \quad m = u - u_{xx},
\end{align*}
which is similar to the DP equation and has long been a standard for studying peakon movement, integrable structures and singularity formation in nonlinear dispersive systems \cite{bressan2007global,danchin2001remarks}. Like the CH equation, the DP equation has been extended to a completely integrable hierarchy through a \(3 \times 3\) matrix Lax pair, which allows for an involutive representation of solutions under a Neumann constraint on a symplectic submanifold \cite{qiao2004integrable}; further studies have confirmed the existence of algebro-geometric solutions for this \(3 \times 3\) integrable system \cite{hou2013algebro}, building on analogous results for the CH equation's \(2 \times 2\) Lax pair formulation \cite{holm1998euler}. Extensive research has been devoted to the Cauchy problem and initial boundary value problem of the DP equation, as documented in \cite{coclite2006well,escher2009initial,yin2003cauchy,yin2004global}.

For $v \equiv u,$ the system \eqref{(1.1)} becomes the following Novikov equation which was proposed in \cite{novikov2009generalizations} 
\begin{align*}
	m_t + u^2 m_x + 3u u_x m = 0, \quad m = u - u_{xx}.
\end{align*}
Notably, the Novikov equation is an integrable peakon system woth cubic nonlineaity admitting peakon solutions. Furthermore, extensive investigations have been conducted on its well-posedness, blow-up phenomena, and ill-posedness, as documented in \cite{himonas2012cauchy,himonas2013holder,Li2025}.

In view of the significance of conservation laws in studying the blow-up phenomena and global existence, we shall summarize two useful conservation laws which
are for the DP and Novikov equations in the Table 1. But unfortunately, the below two conservation laws (the $H^1\left(\mathbb{R}\right)$ and $L^1\left(\mathbb{R}\right)$ norms) are unavailable to the system \eqref{(1.1)}.
\begin{table}[h]
	\centering
	\caption{Conservation laws}
	\begin{tabular}{|c|c|c|c|}
		\hline
		\diagbox{Conservation laws}{Equations} & DP equation & Novikov equation \\
		\hline
		Is $\int_{\mathbb{R}} (um)(t,x)dx$ conserved?  & no & yes \\
		\hline
		Is $\int_{\mathbb{R}} m(t,x)dx$ conserved?  & yes & no \\
		\hline
	\end{tabular}
	\label{tab:conservation_laws}
\end{table}\\

Finally,  consider the following $b$-family of two-component system with cubic nonlinearity proposed in \cite{yan2018new}
\begin{align}\label{(1.2)}
	\begin{cases}m_t+u\upsilon m_x+b\upsilon u_xm=0,\\n_t+u\upsilon n_x+bu\upsilon_xn=0,\\m=u-u_{xx},\quad n=v-v_{xx},\end{cases}
\end{align}
which is the generalization of the Geng-Xue system \eqref{(1.1)}. And Yan et al. \cite{yan2018new} have already established its local well-posedness, blow-up criteria and so on.

  \subsection{Main result}
    Now let us set up the Cauchy problem. Using the Green function $p(x)\triangleq\frac{1}{2}e^{-|x|},x\in\mathbb{R}$ and the identity $D^{-2}f=p*f$ for all $f\in L^2(\mathbb{R})$ with $D^{s}= (1-\partial_x^2)^{\frac{s}{2}}$,  we can rewrite system \eqref{(1.1)} with the initial data $(u_0,v_0)$ as the following form
    \begin{align}\label{(1.3)}
    	\begin{cases}
    		u_t + uvu_x + p * (3 u v u_x + 2uv_x u_{xx} + 2u_x^2 v_x + uv_{xx}u_x) = 0, \\
    		v_t + vuv_x + p * (3 v u v_x + 2vu_x v_{xx} + 2v_x^2 u_x + vu_{xx}v_x) = 0, \\
    		u(0, x) = u_0(x), \quad v(0, x) = v_0(x).
    	\end{cases}
    \end{align}
	We then have the following result for the Geng-Xue system.
	\begin{theo}\(\text{(Blow-up criteria of the Geng-Xue system)}\)\label{Theorem 1.1}
		Let $(u_0,v_0) \in B^2_{2,1}(\mathbb{R}) \times B^2_{2,1}(\mathbb{R})$ and $T^*$ be the maximal existence time of the solution $(u,v)$ to the system \eqref{(1.3)}. If $T < \infty$, then
		\[
		\int_{0}^{T} \| u\|_{W^{1,\infty}} \| v\|_{W^{1,\infty}} dt = \infty.
		\]
	\end{theo}  
	
	\begin{rema}\label{Remark 1.2}
		For the Sobolev space, we have a similar result as follows, which covers the corresponding result in \cite{qiao2019persistence}. Let $(u_0,v_0) \in H^s(\mathbb{R}) \times H^s(\mathbb{R})$ with $s > 2$ and $T$ be the maximal existence time of the solution $(u,v)$ to the system \eqref{(1.3)}. If $T < \infty$, then
		\[
		\int_{0}^{T} \| u\|_{W^{1,\infty}} \| v\|_{W^{1,\infty}} dt = \infty.
		\]
	\end{rema}
  \begin{theo}\(\text{(Blow-up of the Geng-Xue system)}\)\label{Theorem 1.3}
  	Assume that $u_0 \in W^{1,1}\left(\mathbb{R}\right) \cap H^s\left(\mathbb{R}\right)$ and $v_0 \in H^s\left(\mathbb{R}\right)$ with $s > \frac{5}{2}$. Let $T^*$ be the maximal existence time of the corresponding strong solution $u$ to system \eqref{(1.3)}. Fixed some $T_0 \in (0,\frac{v_0\left(x_0\right)}{80\left(\Vert u_0\Vert_{W^{1,1}}+\Vert n_0\Vert_{L^{\infty}}\right)^3})$ and there exists a point $x_0 \in \mathbb{R}$ such that
  	\begin{align*}
  		v_0(x_0) \ge 0,
  	\end{align*}
  	and
  	\begin{align}\label{(1.4)}
  		u_{0,x}\left(x_0\right) \le \frac{1+e^{\frac{\sqrt{b_1v_0^3\left(x_0\right)}}{80\left(\| u_0\|_W^{1,1} + \| n_0\|_{L^{\infty}}\right)^3}}}{1-e^{\frac{\sqrt{b_1v_0^3\left(x_0\right)}}{80\left(\| u_0\|_W^{1,1} + \| n_0\|_{L^{\infty}}\right)^3}}}\sqrt{\frac{b_1}{v_0\left(x_0\right)}},
  	\end{align}
  	with
  	\begin{align*}
  		b_1=\frac{4}{v_0\left(x_0\right)}\left(\Vert n_0\Vert_{L^{\infty}}+\Vert u_0\Vert_{W^{1,1}}\right)^4+14\left(\Vert n_0\Vert_{L^{\infty}}+\Vert u_0\Vert_{W^{1,1}}\right)^3.
  	\end{align*}
  	Then the strong solution $(u,v)$ blows up in finite time with $T^* \le T_0.$ 
  \end{theo}

 \subsection{Application to the two-component b-family system with cubic nonlinearity}
 Now similar to system \eqref{(1.1)}, we can rewrite the Cauchy problem of system \eqref{(1.2)} as follows 
 \begin{align}
 	\begin{cases}\label{(1.5)}
 		u_t + uvu_x + p * (b u v u_x + (3-b)u_xvu_{xx} + 2uv_x u_{xx} + 2u_x^2 v_x + uv_{xx}u_x) = 0, \\
 		v_t + vuv_x + p * (b v u v_x + (3-b)v_xuv_{xx} + 2vu_x v_{xx} + 2v_x^2 u_x + vu_{xx}v_x) = 0, \\
 		u(0, x) = u_0(x), \quad	v(0, x) = v_0(x),
 	\end{cases}
 \end{align}
 with  $p(x)\triangleq\frac{1}{2}e^{-|x|}$.
We then get the following result for the two-component $b$-family system with cubic nonlinearity.
 \begin{theo}\label{Theorem 1.4}
 	Assume that $u_0 \in W^{1,r}\left(\mathbb{R}\right) \cap H^s\left(\mathbb{R}\right)$ and $v_0 \in H^s\left(\mathbb{R}\right)$ with $s > \frac{5}{2}$ and $r \ge 2$. Let $T_3$ be the maximal existence time of the corresponding strong solution $(u,v)$ to system \eqref{(1.5)} with $b= 1+\frac{2}{r}$. Fixed some $T_6 \in \left(0,T_5\right)$ and there exist a point $x_0 \in \mathbb{R}$ such that
 	\begin{align*}
 		v_0(x_0) > 0,
 	\end{align*}
 	and
 	\begin{align}\label{(1.6)}
 		u_{0,x}\left(x_0\right) \le \frac{1+e^{\sqrt{c_2v_0(x_0)}T_5}}{1-e^{\sqrt{c_2v_0(x_0)}T_5}}\sqrt{\frac{c_2}{v_0\left(x_0\right)}},
 	\end{align}
 	with
 	\begin{align*}
 		c_2:=\frac{4}{v_0\left(x_0\right)}r^{\frac{2}{r}}\left(\Vert n_0\Vert_{L^{\infty}}+\Vert u_0\Vert_{W^{1,r}}^r\right)^{\frac{2r+2}{r}}+14\left(\frac{r^2}{2}\right)^{\frac{1}{r}}\left(\frac{r-1}{r}\right)^{\frac{r-1}{r}}\left(\Vert n_0\Vert_{L^{\infty}}+\Vert u_0\Vert_{W^{1,r}}^r\right)^{\frac{r+2}{r}}.
 	\end{align*}
 	Then the strong solution $(u,v)$ blows up in finite time with $T_3 \le T_6.$ 
 \end{theo}
 
 \begin{rema}\label{Remark 1.5}
 	Notably, we find that the index $r$ depends on the constant $b$.
 \end{rema}
 \begin{rema}\label{Remark 1.6}
 We can get a similar well-posedness result for the system \eqref{(1.5)} as Lemma \ref{Lemma 2.6} by using the method presented in \cite{himonas2016novikov}.	If $s > \frac{3}{2}$ and $(u_0, v_0)$ on the line or on the circle, then there exists $T > 0,$ and a unique solution $(u,v) \in C([0,T];H^s \times H^s)$ of the Cauchy problem \eqref{(1.5)} satisfying the following size estimate and lifespan
 \begin{align*}
 	\Vert (u(t),v(t))\Vert_{H^s} \le \sqrt{2}\Vert (u_0,v_0)\Vert_{H^s}, \quad \text{for} \quad 0\le t \le T=\frac{1}{4c_s\Vert (u_0,v_0)\Vert_{H^s}^2},
 \end{align*}
 where $c_s$ is a constant depending on $s$.
 \end{rema}
 Moreover, we can establish a similar blow-up criteria as Theorem \ref{Theorem 1.1} and \ref{Remark 1.2}.
 \begin{rema}\label{Remark 1.7}
 		Let $(u_0,v_0) \in B^2_{2,1}(\mathbb{R}) \times B^2_{2,1}(\mathbb{R})$ and $T$ be the maximal existence time of the solution $(u,v)$ to the system \eqref{(1.5)}. If $T < \infty$, then
 	\[
 	\int_{0}^{T} \| u\|_{W^{1,\infty}} \| v\|_{W^{1,\infty}} dt = \infty.
 	\]
 \end{rema}
 \begin{rema}\label{Remark 1.8}
 	Let $(u_0,v_0) \in H^s(\mathbb{R}) \times H^s(\mathbb{R})$ with $s > 2$ and $T$ be the maximal existence time of the solution $(u,v)$ to the system \eqref{(1.5)}. If $T < \infty$, then
 	\[
 	\int_{0}^{T} \| u\|_{W^{1,\infty}} \| v\|_{W^{1,\infty}} dt = \infty.
 	\]
 \end{rema}

  The remainder of this paper is structured as follows. In Section 2, we review relevant existing results. Section 3 focuses on the blow-up criteria of the Geng-Xue  system. Section 4 is dedicated to proving a blow-up theorem for the Geng-Xue system. In Section 5, we generalize the methodology developed in Section 3 to derive blow-up results for the $b$-family of two-component cubic nonlinear systems.


\section{Preliminaries}
  ~~~~In this section, we will recall some facts about the Geng-Xue system, which will be used in the sequel. Firstly, we present some facts on the Littlewood-Paley decomposition and nonhomogeneous Besov spaces.
  \begin{prop}\cite{BCD}\label{Proposition 2.1}
  	Let $\mathscr{C}$ be an annulus and $\mathscr{B}$ a ball. A constant C exists such that for any $k\in\mathbb{N}$, $1\leq p\leq q\leq\infty$, and any function $u$ of $L^p(\mathbb{R}^d)$, we have
  	$$
  	\mathrm{Supp~}\widehat{u}\subset\lambda\mathscr{B}\Longrightarrow\|D^ku\|_{L^q}=\sup_{|\alpha|=k}\|\partial^\alpha u\|_{L^q}\leq C^{k+1}\lambda^{k+d(\frac{1}{p}-\frac{1}{q})}\|u\|_{L^p},
  	$$
  	$$
  	\mathrm{Supp~}\widehat{u}\subset\lambda\mathscr{C}\Longrightarrow C^{-k-1}\lambda^k\|u\|_{L^p}\leq\|D^ku\|_{L^p}\leq C^{k+1}\lambda^k\|u\|_{L^p}.
  	$$
  \end{prop}
  
  \begin{prop}\cite{BCD}\label{Proposition 2.2}
  	Let $\mathcal{B}=\{\xi\in\mathbb{R}^d: |\xi|\leq \frac 43\}$ and $\mathcal{C}=\{\xi\in\mathbb{R}^d: \frac 34 \leq|\xi|\leq \frac 83\}$. There exist two smooth, radial functions $\chi$ and $\varphi$, valued in the interval $[0,1]$, belonging respectively to $\mathcal{D(B)}$ and $\mathcal{D(C)}$, such that
  	$$
  	\forall\ \xi\in\mathbb{R}^{d},\ \chi(\xi)+\sum_{j\geq0}\varphi(2^{-j}\xi)=1, 
  	$$
  	$$
  	j\geq1\Rightarrow\mathrm{Supp~}\chi\cap\mathrm{Supp~}\varphi(2^{-j}\cdot)=\emptyset,
  	$$
  	Let $u\in\mathcal{S}'$. Defining
  	$$
  	\Delta_{j}u\triangleq0\ \text{if}\ j\leq-2,\ \Delta_{-1}u\triangleq\chi(D)u=\mathcal{F}^{-1}(\chi\mathcal{F}u),
  	$$
  	$$
  	\Delta_{j}u\triangleq\varphi(2^{-j}D)u=\mathcal{F}^{-1}(\varphi(2^{-j}\cdot)\mathcal{F}u)\ \text{if}\ j\geq0,
  	$$
  	$$
  	S_j u\triangleq\sum_{j^{\prime}\leq j-1}\Delta_{j^{\prime}}u,
  	$$
  	we have the following Littlewood-Paley decomposition
  	$$u=\sum_{j\in\mathbb{Z}}\Delta_{j}u\quad in\ \mathcal{S}'.$$
  \end{prop}
  
  \begin{defi}\cite{BCD}\label{Defintion 2.3}
  	Let $s\in\mathbb{R}$ and $(p,r)\in[1,\infty]^2$. The nonhomogeneous Besov
  	space $B^s_{p,r}$ consists of all $u\in\mathcal{S}'$ such that
  	$$
  	\|u\|_{B_{p,r}^s}\triangleq\left\|(2^{js}\|\Delta_ju\|_{L^p})_{j\in\mathbb{Z}}\right\|_{\ell^r(\mathbb{Z})}<\infty.
  	$$
  \end{defi}
  
  \begin{defi}\cite{BCD}\label{Defintion 2.4}
  	Considering $u,v\in\mathcal{S}'$, we have the following Bony decomposition
  	$$uv=T_u v+T_v u+R(u,v),$$
  	where
  	$$
  	T_u v=\sum_j S_{j-1}u\Delta_j v,\ R(u,v)=\sum_{|k-j|\leq1}\Delta_ku\Delta_jv.
  	$$
  \end{defi}
  
  \begin{lemm}\label{Lemma 2.5}\cite{BCD}
  	{\rm (1)} $\forall t<0,s\in\mathbb{R},u\in B^t_{p,r_1}\cap L^\infty,v\in B^s_{p,r_2}$ with $\frac 1 r=\frac 1 {r_1}+\frac 1 {r_2}$, then
  	$$\|T_u v\|_{B^s_{p,r_2}}\leq C\|u\|_{L^\infty}\|v\|_{B^s_{p,r_2}}$$
  	or
  	$$\|T_u v\|_{B^{s+t}_{p,r}}\leq C\|u\|_{B^t_{\infty,r_1}}\|v\|_{B^s_{p,r_2}}$$
  	{\rm (2)} $\forall s_1,s_2\in\mathbb{R},1\leq p_1,p_2,r_1,r_2\leq\infty$, with $\frac 1 p=\frac 1 {p_1}+\frac 1 {p_2}\leq1,\frac 1 r=\frac 1 {r_1}+\frac 1 {r_2}\leq1$. Then $\forall (u,v)\in B^{s_1}_{p_1,r_1}\times B^{s_2}_{p_2,r_2}$, if $s_1+s_2>0$
  	$$\|R(u,v)\|_{B^{s_1+s_2}_{p,r}}\leq C\|u\|_{B^{s_1}_{p_1,r_1}}\|v\|_{B^{s_2}_{p_2,r_2}}$$
  	If $r=1$ and $s_1+s_2=0$,
  	$$\|R(u,v)\|_{B^0_{p,\infty}}\leq C\|u\|_{B^{s_1}_{p_1,r_1}}\|v\|_{B^{s_2}_{p_2,r_2}}$$
  \end{lemm}
  
   Secondly, we state the following local well-posedness for the Cauchy problem \eqref{(1.3)}.
  \begin{lemm}\cite{himonas2016novikov}\label{Lemma 2.6}
  	If $s > \frac{3}{2}$ and $(u_0, v_0)$ on the line or on the circle, then there exists $T > 0,$ and a unique solution $(u,v) \in C([0,T];H^s \times H^s)$ of the Cauchy problem \eqref{(1.3)} satisfying the following size estimate and lifespan
  	\begin{align*}
  		\Vert (u(t),v(t))\Vert_{H^s} \le \sqrt{2}\Vert (u_0,v_0)\Vert_{H^s}, \quad \text{for} \quad 0\le t \le T=\frac{1}{4c_s\Vert (u_0,v_0)\Vert_{H^s}^2},
  	\end{align*}
  	where $c_s$ is a constant depending on $s$.
  \end{lemm}
  
 Finally, we need the following useful results, which will be the key to proving our main ideas.

\begin{prop}\label{Proposition 2.7}
	Let $u \in W^{1,r}\left(\mathbb{R}\right),$ with $1 \le r \le \infty,$ then we have
	\begin{align*}
		2|u|^r \le r\Vert u\Vert^r_{W^{1,r}}.
	\end{align*}
\end{prop}
\begin{proof}
	It is easy to see that when $r = 1,$  
	\begin{align*}
		2|u| \le \int_{x}^{\infty} |u_x| ds + \int_{-\infty}^{x} |u_x| ds \le 2\| u\|_{W^{1,1}}
	\end{align*}
	For $r>1,$ consider
	\begin{align*}
		2|u|^r =&\lim_{\epsilon \rightarrow 0} 2u^2(u^2 + \epsilon)^{\frac{r-2}{2}} 
		\\=&\lim_{\epsilon \rightarrow 0} \int_{-\infty}^{x}  2uu_x(u^2 + \epsilon)^{\frac{r-2}{2}} + (r-2)u^2(u^2 + \epsilon)^{\frac{r-4}{2}}uu_x ds\\
		& -\int_{x}^{\infty} 2uu_x(u^2 + \epsilon)^{\frac{r-2}{2}} + (r-2)u^2(u^2 + \epsilon)^{\frac{r-4}{2}}uu_xds
		\\= & \int_{-\infty}^{x} ruu_x|u|^{r-2} ds - \int_{x}^{\infty} ruu_x|u|^{r-2}ds.
	\end{align*}
	Thus, one gets
	\begin{align*}
		2|u|^r \le r\Vert u_xu^{r-1}\Vert_{L^1}.
	\end{align*}
	Hence, we have
	\begin{align*}
		2|u|^r &\le r\Vert u_x\Vert_r\Vert u^{r-1}\Vert_{\frac{r}{r-1}}= r\Vert u_x\Vert_r\Vert u\Vert_r^{r-1}.
	\end{align*}
	Therefore, we deduce that
	\begin{align*}
		2|u|^r \le r\Vert u\Vert^r_{W^{1,r}}.
	\end{align*}
\end{proof}
  
 \begin{lemm}\cite{chen2016blowup}\label{Lemma 2.8}
 	Let $f\in C^1(\mathbb{R}), a > 0, b > 0$ and $f(0) < -\sqrt{\frac{b}{a}}.$
 	If \begin{align*}
 		f^{\prime}(t)\leq-af^2(t)+b,
 	\end{align*}
 	then
 	\begin{align*}
 		f(t)\to-\infty\quad\mathrm{as}\quad t\to t^*\leq\frac{1}{2\sqrt{ab}}\ln\left(\frac{f(0)-\sqrt{\frac{b}{a}}}{f(0)+\sqrt{\frac{b}{a}}}\right).
 	\end{align*}
 \end{lemm}
 
 \section{Blow-up criteria}
 \textbf{Proof of the Theorem \ref{Theorem 1.1}}:\\
  Together with \eqref{(1.3)}, it is easy to check that
 \begin{align}\label{(3.1)}
 	\| u_t\|_{B^2_{2,1}} \le \|vuu_x\|_{B^2_{2,1}} + \| F\|_{B^2_{2,1}},  
 \end{align}
 where we denote $F$ as follows
 \begin{align*}
 	F:=p * (3 u v u_x + 2uv_x u_{xx} + 2u_x^2 v_x + uv_{xx}u_x).
 \end{align*}
 Regarding the nonlocal term of \eqref{(3.1)}, we have
 \begin{align*}
 	\| F\|_{B^2_{2,1}} \le c\left(\| vuu_x\|_{B^{0}_{2,1}} + \| uu_xv_x\|_{B^{1}_{2,1}} + \| v_xu_x^2\|_{B^{0}_{2,1}} + \| uv_xu_{xx}\|_{B^{0}_{2,1}}\right).
 \end{align*}
 By the Bony decomposition and Lemma \ref{Lemma 2.5}, one gets that
 \begin{align*}
 	 &\| T_{uv_x} u_{xx}\|_{B^{0}_{2,1}} \le c\| uv_x\|_{L^{\infty}}\| u_{xx}\|_{B^0_{2,1}} \le c\| u\|_{L^{\infty}}\|v_x\|_{L^{\infty}}\| u\|_{B^2_{2,1}},
 	 \\& \| T_{u_{xx}} uv_x\|_{B^0_{2,1}} \le c\| u_{xx}\|_{B^{-1}_{\infty,\infty}}\| uv_x\|_{B^1_{2,1}} \le c\| u_x\|_{L^{\infty}}\|v_x\|_{L^{\infty}}\| u\|_{B^2_{2,1}},
 	 \\& \| R(u_{xx},uv_x)\|_{B^0_{2,1}}  \le c\| u_{xx}\|_{B^0_{2,1}}\| uv_x\|_{L^{\infty}} \le c\| u\|_{L^{\infty}}\|v_x\|_{L^{\infty}}\| u\|_{B^2_{2,1}}.
 \end{align*}
 We then obtain that
 \begin{align*}
 	\| uv_xu_{xx}\|_{B^0_{2,1}} \le c\| u\|_{W^{1,\infty}}\|v_x\|_{L^{\infty}}\| u\|_{B^2_{2,1}} \le c\| u\|_{W^{1,\infty}}\| v\|_{W^{1,\infty}}\| u\|_{B^2_{2,1}}.
 \end{align*}
 By a similar way, it is easy to see that
 \begin{align*}
 	&\| F\|_{B^2_{2,1}} \le c\| u\|_{W^{1,\infty}}\| v\|_{W^{1,\infty}}\| u\|_{B^2_{2,1}},
 	\\& \| uvu_x\|_{B^2_{2,1}} \le c\| v\|_{L^{\infty}}\| u_x\|_{L^{\infty}}\| u\|_{B^2_{2,1}} \le c\| u\|_{W^{1,\infty}}\| v\|_{W^{1,\infty}}\| u\|_{B^2_{2,1}}.
 \end{align*}
 Therefore, we have
 \begin{align}\label{(3.2)}
 	\| u_t\|_{B^2_{2,1}} \le c\| u\|_{W^{1,\infty}}\| v\|_{W^{1,\infty}}\| u\|_{B^2_{2,1}}.
 \end{align}
 The analogous inequality for $v$ reads
 \begin{align}\label{(3.3)}
 	\| v_t\|_{B^2_{2,1}} \le c\| u\|_{W^{1,\infty}}\| v\|_{W^{1,\infty}}\| v\|_{B^2_{2,1}}.
 \end{align}
 Adding \eqref{(3.2)} and \eqref{(3.3)}, we deduce that
 \begin{align*}
 	\dfrac{\mathrm{d}\left(\| v\|_{B^2_{2,1}} + \|u\|_{B^2_{2,1}}\right)}{\mathrm{d}t} \le c\| u\|_{W^{1,\infty}}\| v\|_{W^{1,\infty}} \left(\| v\|_{B^2_{2,1}} + \|u\|_{B^2_{2,1}}\right).
 \end{align*}
 Taking advantage of Gronwall's inequality, one gets
 \begin{align}\label{(3.4)}
 	\| v\|_{B^2_{2,1}} + \|u\|_{B^2_{2,1}} \le \left(\| v_0\|_{B^2_{2,1}} + \|u_0\|_{B^2_{2,1}}\right)e^{c\int_{0}^{t} \| u\|_{W^{1,\infty}}\| v\|_{W^{1,\infty}} d\tau}.
 \end{align}
 Hence, if $T < \infty$ satisfies $ \int_{0}^{T} \| u\|_{W^{1,\infty}}\| v\|_{W^{1,\infty}} d\tau < \infty$, then we have
 \[
 \limsup_{t \to T} \left( \| v\|_{B^2_{2,1}} + \|u\|_{B^2_{2,1}} \right) < \infty,
 \]
 which contracts the assumption that $T < \infty$ is the maximal existence time. This completes of the proof of the theorem.
$\hfill\square$

 \section{Blow-up}
 ~~~~In this section, we will construct some blow-up solutions to the system $(1.3)$. To achieve it, we need the following results.

\begin{prop}\label{Proposition 4.1}
	Assume that $n_0 \in L^{\infty}\left(\mathbb{R}\right)$ and $u_0 \in W^{1,1}\left(\mathbb{R}\right).$ Let $T^*$ be the maximal existence time of the corresponding strong solution $\left(u,v\right)$ to system \eqref{(1.3)}. Then we have
\begin{align}\label{(4.1)}
	\Vert n \Vert_{L^{\infty}}+\Vert u\Vert_{W^{1,1}} \le 2\left(\Vert n_0\Vert_{L^{\infty}}+\Vert u_0\Vert_{W^{1,1}}\right).
\end{align}
with
\begin{align*}
	t \le T_1=\frac{1}{80\left(\Vert u_0\Vert_{W^{1,1}}+\Vert n_0\Vert_{L^{\infty}}\right)^{2}}.
\end{align*} 
\end{prop}

\begin{proof}
	The characteristics $q(t,x)$ associated the Geng-Xue system \eqref{(1.1)}, which is given as follows
	\begin{align}\label{(4.2)}
		\begin{cases}\dfrac{\mathrm{d}}{\mathrm{d}t}q(t,x)=(uv)(t,q(t,x)),&(t,x)\in[0,T^*)\times\mathbb{R},\\q(0,x)=x,&x\in\mathbb{R}.\end{cases}
	\end{align}
	According to the classical theory of ordinary differential equations, we get the above equation has an unique solution
	\begin{align*}
		q(t,x) \in C^1\left([0,T^*) \times \mathbb{R},\mathbb{R}\right).
	\end{align*}
	Moreover, the map $x \to q(t,x)$ is an increasing diffeomorphism. In this way, we have
	\begin{align*}
		\frac{\mathrm{d}n\left(t,q\left(t,x\right)\right)}{\mathrm{d}t} &= n_t\left(t,q\left(t,x\right)\right)+n_x\left(t,q\left(t,x\right)\right) \\&= \left(n_t+n_xuv\right)\left(t,\left(t,x\right)\right) \\&= -3v_xun.
	\end{align*}
	Since $u=(1-\partial_{x}^{2})^{-1}m=p*m$ with $p(x)\triangleq\frac{1}{2}e^{-|x|}, u_{x}=(\partial_{x}p)*m.$ and $\Vert p\Vert_{L^1}=\Vert \partial_xp\Vert_{L^1}=1$, together with the Young's inequality, for any $s\in\mathbb{R},$ we have
	\begin{align*}
		\Vert u\Vert_{L^{\infty}} \le \Vert m\Vert_{L^{\infty}},\quad \Vert u_{x}\Vert_{L^{\infty}} \le \Vert m\Vert_{L^{\infty}},
	\end{align*}
	thus,
	\begin{align*}
		\Vert u_{xx}\Vert_{L^{\infty}} \le \Vert m\Vert_{L^{\infty}}+\Vert m\Vert_{L^{\infty}}=2\Vert m\Vert_{L^{\infty}}.
	\end{align*}
	In the similar way, one gets that
	\begin{align*}
		\Vert v\Vert_{L^{\infty}} \le \Vert n\Vert_{L^{\infty}}, \Vert v_{x}\Vert_{L^{\infty}} \le \Vert n\Vert_{L^{\infty}},\Vert v_{xx}\Vert_{L^{\infty}} \le 2\Vert n\Vert_{L^{\infty}}.
	\end{align*}
	It is easy to check that
	\begin{align*}
		|uv_xn| \le |u|\Vert v_x\Vert_{L^{\infty}}\Vert n\Vert_{L^{\infty}}.
	\end{align*}
	By Proposition \ref{Proposition 2.7}, we obtain that
	\begin{align*}
		|uv_xn| \le \frac{\Vert u\Vert_{W^{1,1}}\Vert v_x\Vert_{L^{\infty}}\Vert n\Vert_{L^{\infty}}}{2}.
	\end{align*}
	Then we have
	\begin{align*}
		\left|\frac{\mathrm{d} n\left(t,q\left(t,x\right)\right)}{\mathrm{d}t}\right| \le \frac{3}{2}\Vert u\Vert_{W^{1,1}}\Vert v_x\Vert_{L^{\infty}}\Vert n\Vert_{L^{\infty}}.
	\end{align*}
	Thus,
	\begin{align*}
		-\frac{3}{2}\Vert u\Vert_{W^{1,1}}\Vert v_x\Vert_{L^{\infty}}\Vert n\Vert_{L^{\infty}} \le \frac{\mathrm{d}n\left(t,q\left(t,x\right)\right)}{\mathrm{d}t} \le \frac{3}{2}\Vert u\Vert_{W^{1,1}}\Vert v_x\Vert_{L^{\infty}}\Vert n\Vert_{L^{\infty}}.
	\end{align*} 
	Integrating the above inequity with respect to $t$, we have 
	\begin{align*}
		-\frac{3}{2}\int_{0}^{t} \Vert n\left(s\right)\Vert^2_{L^{\infty}}\Vert u\left(s\right)\Vert_{W^{1,1}}ds+\Vert n_0\Vert_{L^{\infty}} \le \Vert n\Vert_{L^{\infty}} \le \frac{3}{2}\int_{0}^{t} \Vert n\left(s\right)\Vert^2_{L^{\infty}}\Vert u\left(s\right)\Vert_{W^{1,1}}ds+\Vert n_0\Vert_{L^{\infty}}. 
	\end{align*}
	Since $q\left(t,x\right)$ is a diffeomorphism of $\mathbb{R},$ we have
	\begin{align}\label{(4.3)}
		\Vert n\Vert_{L^{\infty}} \le \frac{3}{2}\int_{0}^{t} \Vert n\left(s\right)\Vert^2_{L^{\infty}}\Vert u\left(s\right)\Vert_{W^{1,1}}ds+\Vert n_0\Vert_{L^{\infty}}. 
	\end{align}
	Noting the system $\left(1.3\right)$ and differentiating the system $\left(1.3\right)_1$ to $x$, we infer that
	\begin{align*}
		&u_t+uvu_x+p*\left(3vuu_x-uu_xv_{xx}\right)+2p_x*uv_xu_x=0, \\& u_{xt}+vu_x^2-uv_xu_x+vuu_{xx}+p_x*\left(3vuu_x-uu_xv_{xx}\right)+2p*uv_xu_x=0.
	\end{align*}
	Since
	\begin{align*}
		\Vert u_t\Vert_{L^1}=\Vert uvu_x+p*\left(3vuu_x-uu_xv_{xx}\right)+2p_x*uv_xu_x\Vert_{L^1},
	\end{align*}
	applying the Young's inequality, one gets
	\begin{align*}
		&\Vert uvu_x\Vert_{L^1} \le \Vert u\Vert_{L^{\infty}}\Vert v\Vert_{L^{\infty}}\Vert u_x\Vert_{L^1},
		\\&3\Vert p*uvu_x\Vert_{L^1} + \Vert p*uv_{xx}u_x\Vert_{L^1} \le \left(3\Vert v\Vert_{L^{\infty}} + \Vert v_{xx}\Vert_{L^{\infty}}\right)\Vert p\Vert_{L^1}\Vert u\Vert_{L^{\infty}}\Vert u_x\Vert_{L^1},
		\\&\Vert p_x*uv_xu_x\Vert_{L^1} \le \Vert p_x\Vert_{L^1}\Vert u\Vert_{L^{\infty}}\Vert v_x\Vert_{L^{\infty}}\Vert u_x\Vert_{L^1}.
	\end{align*}
	According to Proposition \ref{Proposition 2.7}, we have
	\begin{align*}
		\frac{\mathrm{d}\Vert u\Vert_{L^1}}{\mathrm{d}t} \le \Vert u_t\Vert_{L^1} \le 4\Vert n\left(s\right)\Vert_{L^{\infty}}\Vert u\left(s\right)\Vert_{W^{1,1}}^2.
	\end{align*}
	Integrating the above inequity with respect to $t$, we thus get
	\begin{align}\label{(4.4)}
		\Vert u\Vert_{L^1} \le \Vert u_0\Vert_{L^1}+4\int_{0}^{t} \Vert n\left(s\right)\Vert_{L^{\infty}}\Vert u\left(s\right)\Vert_{W^{1,1}}^2ds.
	\end{align}
	As
	\begin{align*}
		\Vert u_x\Vert_{L^1}=\lim_{\epsilon \rightarrow 0} \langle u_x,u_x\left(u_x^2+\epsilon\right)^{-\frac{1}{2}}\rangle.
	\end{align*}
	Performing integration by parts, we deduce that
	\begin{align*}
		\lim_{\epsilon \rightarrow 0} \int_{\mathbb{R}} uvu_{xx}u_x\left(u_x^2+\epsilon\right)^{-\frac{1}{2}} dx &= \lim_{\epsilon \rightarrow 0} \int_{\mathbb{R}} uvd\left(u_x^2+\epsilon\right)^{\frac{1}{2}} 
		\\&= \lim_{\epsilon \rightarrow 0} -\int_{\mathbb{R}} \left(u_x^2+\epsilon\right)^{\frac{1}{2}}\left(uv_x+u_xv\right)dx 
		\\&= -\int_{\mathbb{R}} \left(u_x^2\right)^{\frac{1}{2}}\left(uv_x+u_xv\right)dx.
	\end{align*}
	We then have
	\begin{align*}
		\Vert u_{xt}\Vert_{L^1} \le \Vert p_x*\left(3uvu_x-uu_xv_{xx}\right)\Vert_{L^1}+2\Vert p*uv_xu_x\Vert_{L^1}+2\Vert uu_xv_x\Vert_{L^1}.
	\end{align*}
	Now applying the Young's inequality, one gets
	\begin{align*}
		&\Vert uv_xu_x\Vert_{L^1} \le \Vert u\Vert_{L^{\infty}}\Vert v_x\Vert_{L^{\infty}}\Vert u_x\Vert_{L^1},
		\\&\Vert p*uv_xu_x\Vert_{L^1} \le \Vert p\Vert_{L^1}\Vert u\Vert_{L^{\infty}}\Vert v_x\Vert_{L^{\infty}}\Vert u_x\Vert_{L^1},
		\\&3\Vert p_x*uvu_x\Vert_{L^1} + \Vert p_x*uv_{xx}u_x\Vert_{L^1} \le  \left( 3\Vert v\Vert_{L^{\infty}} + \Vert v_{xx}\Vert_{L^{\infty}} \right)\Vert p_x\Vert_{L^1}\Vert u\Vert_{L^{\infty}}\Vert u_x\Vert_{L^1}.
	\end{align*}
	Hence, by Proposition \ref{Proposition 2.7}, we obtain that
	\begin{align*}
		\frac{\mathrm{d}\Vert u_x\Vert_{L^1}}{\mathrm{d}t} \le \Vert u_{xt}\Vert_{L^1} \le \frac{9}{2}\Vert n\Vert_{L^{\infty}}\Vert u\Vert_{W^{1,1}}^2.
	\end{align*}
	Integrating the above inequality with respect to $t$, we have
	\begin{align}\label{(4.5)}
		\Vert u_x\Vert_{L^1} \le \Vert u_{0,x}\Vert_{L^1}+\frac{9}{2}\int_{0}^{t} \Vert n\left(s\right)\Vert_{L^{\infty}}\Vert u\left(s\right)\Vert_{W^{1,1}}^2ds.
	\end{align}
	Now using \eqref{(4.4)} and \eqref{(4.5)}, we then get
	\begin{align}\label{(4.6)}
		\Vert u\Vert_{W^{1,1}} \le \Vert u_{0,x}\Vert_{L^1}+\frac{17}{2}\int_{0}^{t} \Vert n\left(s\right)\Vert_{L^{\infty}}\Vert u\left(s\right)\Vert_{W^{1,1}}^2ds.
	\end{align}
	It then follows from \eqref{(4.3)} and \eqref{(4.6)}, we have
	\begin{align*}
		\Vert n\Vert_{L^{\infty}}+\Vert u\Vert_{W^{1,1}} \le \Vert n_0\Vert_{L^{\infty}}+\Vert u_0\Vert_{W^{1,1}}+\int_{0}^{t}\frac{3}{2} \Vert n\left(s\right)\Vert^2_{L^{\infty}}\Vert u\left(s\right)\Vert_{W^{1,1}}+\frac{17}{2}\Vert n\left(s\right)\Vert_{L^{\infty}}\Vert u\left(s\right)\Vert_{W^{1,1}}^2ds. \
	\end{align*}
	It is easy to see that
	\begin{align*}
		\Vert n\Vert_{L^{\infty}}+\Vert u\Vert_{W^{1,1}} \le \Vert n_0\Vert_{L^{\infty}}+\Vert u_0\Vert_{W^{1,1}}+10\int_{0}^{t} \left(\Vert n\Vert_{L^{\infty}}+\Vert u\Vert_{W^{1,1}}\right)^3ds.
	\end{align*}
	Now we obtain
	\begin{align*}
		\Vert u\Vert_{W^{1,1}}+\Vert n\Vert_{L^{\infty}} \le 2\left(\Vert u_0\Vert_{W^{1,1}}+\Vert n_0\Vert_{L^{\infty}}\right),
	\end{align*}
	with
	\begin{align*}
		t \le T_1=\frac{1}{80\left(\Vert u_0\Vert_{W^{1,1}}+\Vert n_0\Vert_{L^{\infty}}\right)^{2}}.
	\end{align*}
	The proof is therefore complete.
	
\end{proof}

 \begin{prop}\label{Proposition 4.2}
 	Assume that $u_0 \in W^{1,1}\left(\mathbb{R}\right)$, $v_0 \in L^{\infty}\left(\mathbb{R}\right)$ and there exists a point $x_0$ such that $v_0\left(x_0\right) > 0.$ Let $T_0$ be the maximal existence time of the corresponding strong solution $\left(u,v\right)$ to system \eqref{(1.3)}. Then we have
 	\begin{align*}
 		v \ge \frac{v_0\left(x_0\right)}{2},
 	\end{align*}
 	with
 	\begin{align*}
 		t \le T_2=\frac{v_0\left(x_0\right)}{80\left(\Vert u_0\Vert_{W^{1,1}}+\Vert n_0\Vert_{L^{\infty}}\right)^3} \le T_1.
 	\end{align*} 
 \end{prop}
 \begin{proof}
 	Consider the system \eqref{(1.3)} along the characteristics $q\left(t,x\right)$, we then have
 	\begin{align*}
 		v_t+p*\left(3uvv_x+u_xv_x^2+vu_xv_{xx}\right)+p_x*vv_xu_x=0.
 	\end{align*} 
 	By the Young's inequality, one gets
 	\begin{align*}
 		&\Vert p*uv_xv\Vert_{L^{\infty}} \le \Vert p\Vert_{L^{\infty}}\Vert v_x\Vert_{L^{\infty}}\Vert v\Vert_{L^{\infty}}\Vert u\Vert_{L^{1}},
 		\\&\Vert p*u_xv_xv_x\Vert_{L^{\infty}} + \Vert p*u_xv_{xx}v\Vert_{L^{\infty}} \le \left(\Vert v_x\Vert_{L^{\infty}}\Vert v_x\Vert_{L^{\infty}} +\Vert v\Vert_{L^{\infty}}\Vert v_{xx}\Vert_{L^{\infty}} \right)\Vert p\Vert_{L^{\infty}}\Vert u_x\Vert_{L^{1}},
 		\\&\Vert p_x*u_xv_xv\Vert_{L^{\infty}} \le \Vert p_x\Vert_{L^{\infty}}\Vert v_x\Vert_{L^{\infty}}\Vert v\Vert_{L^{\infty}}\Vert u_x\Vert_{L^{1}}.
 	\end{align*}
 	Thus, 
 	\begin{align*}
 		|v_t(t,q(t,x))| \le \frac{7}{2}\Vert n\left(t\right)\Vert^2_{L^{\infty}}\Vert u\left(t\right)\Vert_{W^{1,1}}.
 	\end{align*}
 	Integrating the above inequality with respect to $t$, we have
 	\begin{align*}
 		-\frac{7}{2}\int_{0}^{t} \Vert n\left(s\right)\Vert^2_{L^{\infty}}\Vert u\left(s\right)\Vert_{W^{1,1}}ds+v_0\left(x_0\right) \le  v \le \frac{7}{2}\int_{0}^{t} \Vert n\left(s\right)\Vert^2_{L^{\infty}}\Vert u\left(s\right)\Vert_{W^{1,1}}ds+v_0\left(x_0\right). 
 	\end{align*}
 	Combining the Proposition \ref{Proposition 4.1} and the fact that $q(t,x)$ is diffeomorphism of $\mathbb{R}$, we have
 	\begin{align*}
 		v(t,q(t,x_0)) \ge \frac{v_0\left(x_0\right)}{2},
 	\end{align*}
 	with
 	\begin{align*}
 		t \le T_2=\frac{v_0\left(x_0\right)}{80\left(\Vert u_0\Vert_{W^{1,1}}+\Vert n_0\Vert_{L^{\infty}}\right)^3} \le T_1,
 	\end{align*}
 	which completes the proof of the proposition.
 \end{proof}
 
 Now, we are in a position to show our main theorem. \\
 \textbf{Proof of Theorem \ref{Theorem 1.3}}:\\
  Differentiating the system \eqref{(1.3)} to $x$, we deduce that
 	\begin{align*}
 		u_{xt}+vu_x^2-uv_xu_x+vuu_{xx}+p_x*\left(3vuu_x-uu_xv_{xx}\right)+2p*uv_xu_x=0.
 	\end{align*}
 	It is easy to see that
 	\begin{align*}
 		\left(u_{xt}+vu_x^2-uv_xu_x+p_x*\left(3vuu_x-uu_xv_{xx}\right)+2p*uv_xu_x\right)\left(t,q\left(t,x\right)\right)=0.
 	\end{align*}
 	Then using the Young's inequality, one gets
 	\begin{align*}
 		&\Vert p*uv_xu_x\Vert_{L^{\infty}} \le \Vert p\Vert_{L^{\infty}}\Vert u\Vert_{L^{\infty}}\Vert v_x\Vert_{L^{\infty}}\Vert u_x\Vert_{L^1},
 		\\&3\Vert p_x*uvu_x\Vert_{L^{\infty}} + \Vert p_x*uv_{xx}u_x\Vert_{L^{\infty}} \le \left(3\Vert v\Vert_{L^{\infty}} + \Vert v_{xx}\Vert_{L^{\infty}}\right)\Vert p_x\Vert_{L^{\infty}}\Vert u\Vert_{L^{\infty}}\Vert u_x\Vert_{L^1}.
 	\end{align*}
 	By Proposition \ref{Proposition 4.1} and Proposition \ref{Proposition 4.2}, we have
 	\begin{align*}
 		\frac{\mathrm{d}u_x\left(t,q\left(t,x_0\right)\right)}{\mathrm{d}t} 
 		&\le -\frac{v_0\left(t,x_0\right)}{4}u_x^2 + \frac{1}{v_0\left(x_0\right)}\|n\|_{L^{\infty}}^2\|u\|_{L^{\infty}}^2 + \frac{7}{4}\|n\|_{L^{\infty}}\|u\|_{W^{1,1}}^2 \\
 		&\le -\frac{v_0\left(t,x_0\right)}{4}u_x^2 + \frac{1}{4v_0\left(x_0\right)}\|n\|_{L^{\infty}}^2\|u\|_{W^{1,1}}^2 + \frac{7}{4}\|n\|_{L^{\infty}}\|u\|_{W^{1,1}}^2 \\
 		&\le -\frac{v_0\left(t,x_0\right)}{4}u_x^2 + \frac{4}{v_0\left(x_0\right)}\left(\|n_0\|_{L^{\infty}} + \|u_0\|_{W^{1,1}}\right)^4 + 14\left(\|n_0\|_{L^{\infty}} + \|u_0\|_{W^{1,1}}\right)^3.
 	\end{align*}
 	Then define $f:=u_x\left(t,q\left(t,x_0\right)\right)$, $a:=\frac{v_0(x_0)}{4}$ and $b_1:=\frac{4}{v_0\left(x_0\right)}\left(\Vert n_0\Vert_{L^{\infty}}+\Vert u_0\Vert_{W^{1,1}}\right)^4+14\left(\Vert n_0\Vert_{L^{\infty}}+\Vert u_0\Vert_{W^{1,1}}\right)^3.$
 	Thanks to \eqref{(1.4)}. we thus deduce that
 	\begin{align*}
 		\frac{1}{\sqrt{b_1v_0\left(x_0\right)}}\ln\left(\frac{\sqrt{v_0\left(x_0\right)}f\left(0\right)-\sqrt{b_1}}{\sqrt{v_0\left(x_0\right)}f\left(0\right)+\sqrt{b_1}}\right) \le T_2.
 	\end{align*}
 	Applying Lemma \ref{Lemma 2.8}, we have
 	\begin{align*}
 		\lim_{t\rightarrow T_0}f(t)=-\infty,
 	\end{align*}
 	with
 	\begin{align*}
 		T_0 \le \frac{1}{\sqrt{b_1v_0\left(x_0\right)}}\ln\left(\frac{\sqrt{v_0\left(x_0\right)}f\left(0\right)-\sqrt{b_1}}{\sqrt{v_0\left(x_0\right)}f\left(0\right)+\sqrt{b_1}}\right) \le T_2 \le T_1.
 	\end{align*}
 	which along with Lemma \ref{Lemma 2.8} yields the desired result.
 	$\hfill\square$\\
  
  \begin{rema}
  	For the variable $v$, we can also get a similar result.
  \end{rema}

\section{Application}
~~~~In this section, we study the blow-up and ill-posedness of the following Cauchy problem for the two-component b-family system with cubic nonlinearity. Before proving the blow-up of system \eqref{(1.5)}, we firstly present the following estimates.

\begin{prop}\label{Proposition 5.1}
	Assume that $v_0 \in L^{\infty}\left(\mathbb{R}\right)$ and $u_0 \in W^{1,r}\left(\mathbb{R}\right)$ with $2 \le r < \infty.$ Let $T_3$ be the maximal existence time of the corresponding strong solution $\left(u,v\right)$ to system \eqref{(1.5)} with $b= 1+\frac{2}{r}$. Then we have
	\begin{align}\label{(5.1)}
		\Vert n \Vert_{L^{\infty}}+\Vert u\Vert_{W^{1,r}}^r \le 2\left(\Vert n_0\Vert_{L^{\infty}}+\Vert u_0\Vert^p_{W^{1,r}}\right),
	\end{align}
	with
	\begin{align*}
		t \le T_4=\frac{1}{8\left(r^{\frac{1+r}{r}}+r^{\frac{1}{r}}+r^{\frac{1-r}{r}}+7r^{\frac{1}{r}}\left(r-1\right)^{\frac{r-1}{r}}\right)\left(\Vert u_0\Vert_{W^{1,r}}^r+\Vert n_0\Vert_{L^{\infty}}\right)^{\frac{r+1}{r}}}.
	\end{align*} 
\end{prop}
\begin{proof}
	The characteristics $q(t,x)$ associated the system \eqref{(1.2)}, which is given as follows
	\begin{align}\label{(5.2)}
		\begin{cases}\dfrac{\mathrm{d}}{\mathrm{d}t}q(t,x)=(uv)(t,q(t,x)),&(t,x)\in[0,T^*)\times\mathbb{R},\\q(0,x)=x,&x\in\mathbb{R}.\end{cases}
	\end{align}
	According to the classical theory of ordinary differential equations, we get the above equation has a unique solution
	\begin{align*}
		q(t,x) \in C^1\left([0,T^*) \times \mathbb{R},\mathbb{R}\right).
	\end{align*}
	Moreover, the map $x \to q(t,x)$ is an increasing diffeomorphism. In this way, we have
	\begin{align*}
		\frac{\mathrm{d}n\left(t,q\left(t,x\right)\right)}{\mathrm{d}t} &= n_t\left(t,q\left(t,x\right)\right)+n_x\left(t,q\left(t,x\right)\right) \\&= \left(n_t+n_xuv\right)\left(t,\left(t,x\right)\right) \\&= -\left(1+\frac{2}{r}\right)v_xun.
	\end{align*}
	It is easy to check that
	\begin{align*}
		|uv_xn| \le |u|\Vert v_x\Vert_{L^{\infty}}\Vert n\Vert_{L^{\infty}}.
	\end{align*}
	By Proposition \ref{Proposition 2.7}, we obtain that 
	\begin{align*}
		|uv_xn| \le \left(\frac{r}{2}\right)^{\frac{1}{r}}\Vert u\Vert_{W^{1,r}}\Vert v_x\Vert_{L^{\infty}}\Vert n\Vert_{L^{\infty}}.
	\end{align*}
	We then get
	\begin{align*}
		\left|\frac{\mathrm{d}n\left(t,q\left(t,x\right)\right)}{\mathrm{d}t}\right| \le \left(1+\frac{2}{r}\right)\left(\frac{r}{2}\right)^{\frac{1}{r}}\Vert u\Vert_{W^{1,r}}\Vert v_x\Vert_{L^{\infty}}\Vert n\Vert_{L^{\infty}}.
	\end{align*}
	Thus,
	\begin{align*}
		-\left(1+\frac{2}{r}\right)\left(\frac{r}{2}\right)^{\frac{1}{r}} \Vert u\Vert_{W^{1,r}}\Vert v_x\Vert_{L^{\infty}}\Vert n\Vert_{L^{\infty}} \le \frac{dn\left(t,q\left(t,x\right)\right)}{dt} \le \left(1+\frac{2}{r}\right)\left(\frac{r}{2}\right)^{\frac{1}{r}}\Vert u\Vert_{W^{1,r}}\Vert v_x\Vert_{L^{\infty}}\Vert n\Vert_{L^{\infty}}.
	\end{align*} 
	Integrating the above inequality with respect to $t$, we have 
	\begin{align*}
		&\Vert n\Vert_{L^{\infty}} \ge \Vert n_0\Vert_{L^{\infty}} - \left(1+\frac{2}{r}\right)\left(\frac{r}{2}\right)^{\frac{1}{r}}\int_{0}^{t} \Vert n\left(s\right)\Vert^2_{L^{\infty}}\Vert u\left(s\right)\Vert_{W^{1,r}}ds,
		\\&\Vert n\Vert_{L^{\infty}} \le \Vert n_0\Vert_{L^{\infty}} + \left(1+\frac{2}{r}\right)\left(\frac{r}{2}\right)^{\frac{1}{r}}\int_{0}^{t} \Vert n\left(s\right)\Vert^2_{L^{\infty}}\Vert u\left(s\right)\Vert_{W^{1,r}}ds. 
	\end{align*}
	Since $q\left(t,x\right)$ is a diffeomorphism of $\mathbb{R},$ one gets that
	\begin{align}\label{(5.3)}
		\Vert n\Vert_{L^{\infty}} \le \left(1+\frac{2}{r}\right)\left(\frac{r}{2}\right)^{\frac{1}{r}}\int_{0}^{t} \Vert n\left(s\right)\Vert^2_{L^{\infty}}\Vert u\left(s\right)\Vert_{W^{1,r}}ds+\Vert n_0\Vert_{L^{\infty}}. 
	\end{align}
	Noting system \eqref{(1.5)} and differentiating the system \eqref{(1.5)} to $x$, we deduce that
	\begin{align*}
		u_t+uvu_x+p*\left(\left(1+\frac{2}{r}\right)vuu_x-uu_xv_{xx}-\left(1-\frac{1}{r}\right)v_xu_x^2\right)+p_x*\left(2uv_xu_x+\left(1-\frac{1}{r}\right)vu_x^2\right)=0,  
	\end{align*}
	\begin{align*}
		u_{xt}+\frac{1}{r}vu_x^2-uv_xu_x+vuu_{xx}&+p_x*\left(\left(1+\frac{2}{r}\right)vuu_x-uu_xv_{xx}-\left(1-\frac{1}{r}\right)v_xu_x^2\right)\\&+p*\left(2uv_xu_x+\left(1-\frac{1}{r}\right)vu_x^2\right)=0.
	\end{align*}
	Using the fact that
	\begin{align*}
		\frac{\mathrm{d}\Vert u\Vert_{L^r}^r}{\mathrm{d}t}=r\langle u_t,u|u|^{r-2}\rangle \le r\Vert u_t\Vert_{L^r}\Vert u\Vert_{L^r}^{r-1}.
	\end{align*}
	As we have
	\begin{align*}
		\Vert u_t\Vert_{L^r}= \left\Vert uvu_x+p*\left(\left(1+\frac{2}{r}\right)vuu_x-uu_xv_{xx}-\left(1-\frac{1}{r}\right)v_xu_x^2\right)+p_x*\left(2uv_xu_x+\left(1-\frac{1}{r}\right)vu_x^2\right)\right\Vert_{L^r},
	\end{align*}
	and applying the Young's inequality, one gets
	\begin{align*}
		&\Vert uvu_x\Vert_{L^r} \le \Vert u\Vert_{L^\infty}\Vert v\Vert_{L^\infty}\Vert u_x\Vert_{L^r}, \\
		&\Vert p*u_xv_xu_x\Vert_{L^r} \le \Vert p\Vert_{L^{\frac{r}{r-1}}}\Vert u_x\Vert_{L^r}\Vert v_x\Vert_{L^\infty}\Vert u_x\Vert_{L^r}, \\
		&\left(1+\frac{2}{r}\right)\Vert p*uvu_x\Vert_{L^r} + \Vert p*uv_{xx}u_x\Vert_{L^r} \le \left(\left(1+\frac{2}{r}\right)\Vert v\Vert_{L^\infty} + \Vert v_{xx}\Vert_{L^\infty} \right) \\
		&\qquad \cdot\Vert p\Vert_{L^{\frac{r}{r-1}}}\Vert u\Vert_{L^r}\Vert u_x\Vert_{L^r}, \\
		&2\Vert p_x*uv_xu_x\Vert_{L^r} + \left(1-\frac{1}{r}\right)\Vert p_x*u_xvu_x\Vert_{L^r} \le \left(2\Vert u\Vert_{L^r}\Vert v_x\Vert_{L^\infty} + \left(1-\frac{1}{r}\right)\Vert u_x\Vert_{L^r}\Vert v\Vert_{L^\infty} \right) \\
		&\qquad \cdot\Vert p_x\Vert_{L^{\frac{r}{r-1}}}\Vert u_x\Vert_{L^r}.
	\end{align*}
	Then together with Proposition \ref{Proposition 2.7}, we obtain that
	\begin{align*}
		\frac{\mathrm{d}\Vert u\Vert_{L^r}^r}{\mathrm{d}t} \le r\Vert u_t\Vert_{L^r}\Vert u\Vert_{L^r}^{r-1} \le \left(\frac{r}{2}\right)^{\frac{1}{r}}\left(r+7\left(r-1\right)^{\frac{r-1}{r}}\right)\Vert n\left(s\right)\Vert_{L^{\infty}}\Vert u\left(s\right)\Vert_{W^{1,r}}^{r+1}.
	\end{align*}
	Integrating the above inequality with respect to $t$, we thus derive that
	\begin{align}\label{(5.4)}
		\Vert u\Vert_{L^r}^r \le \Vert u_0\Vert_{L^r}+\left(\frac{r}{2}\right)^{\frac{1}{r}}\left(r+7\left(r-1\right)^{\frac{r-1}{r}}\right)\int_{0}^{t} \Vert n\left(s\right)\Vert_{L^{\infty}}\Vert u\left(s\right)\Vert_{W^{1,r}}^{r+1}ds.
	\end{align}
	As
	\begin{align*}
		\Vert u_x\Vert_{L^r}^r=\lim_{\epsilon \rightarrow 0} \langle u_x,u_x\left(u_x^2+\epsilon\right)^{\frac{r-2}{2}}\rangle,
	\end{align*}
	Therefore, performing integration by parts, we deduce that
	\begin{align*}
		\lim_{\epsilon \rightarrow 0} \int_{\mathbb{R}} uvu_{xx}u_x\left(u_x^2+\epsilon\right)^{\frac{r-2}{2}}dx &= \lim_{\epsilon \rightarrow 0} \frac{1}{r}\int_{\mathbb{R}} uvd\left(u_x^2+\epsilon\right)^{\frac{r}{2}} 
		\\&= \lim_{\epsilon \rightarrow 0} -\frac{1}{r}\int_{\mathbb{R}} \left(u_x^2+\epsilon\right)^{\frac{r}{2}}\left(uv_x+u_xv\right)dx 
		\\&= -\frac{1}{r}\int_{\mathbb{R}} \left(u_x^2\right)^{\frac{r}{2}}\left(uv_x+u_xv\right)dx.
	\end{align*}
	It is easy to see that
	\begin{align*}
		\Vert u_{xt}\Vert_{L^r} 
		&\le \bigg\| -\left(1+\frac{1}{r}\right)uv_xu_x 
		+ p_x*\left(
		\left(1+\frac{2}{r}\right)vuu_x - uu_xv_{xx} - \left(1-\frac{1}{r}\right)v_xu_x^2
		\right) \\
		&\quad + p*\left(
		2uv_xu_x + \left(1-\frac{1}{r}\right)vu_x^2
		\right) \bigg\|_{L^r}.
	\end{align*}
	Then applying the Young's inequality, one gets
	\begin{align*}
		&\Vert uv_xu_x\Vert_{L^r} \le \Vert u\Vert_{L^{\infty}}\Vert v_x\Vert_{L^{\infty}}\Vert u_x\Vert_{L^r},
		\\&2\Vert p*uv_xu_x\Vert_{L^r} + \left(1-\frac{1}{r}\right)\Vert p*u_xvu_x\Vert_{L^r} \le \left(2\Vert u\Vert_{L^{r}}\Vert v_x\Vert_{L^{\infty}} + \left(1-\frac{1}{r}\right)\Vert u_x\Vert_{L^{r}}\Vert v\Vert_{L^{\infty}} \right)\Vert p\Vert_{L^{\frac{r}{r-1}}}\Vert u_x\Vert_{L^r},
		\\&\Vert p_x*u_xv_xu_x\Vert_{L^r} \le \Vert p_x\Vert_{L^{\frac{r}{r-1}}}\Vert u_x\Vert_{L^{r}}\Vert v_x\Vert_{L^{\infty}}\Vert u_x\Vert_{L^r},
		\\&\left(1+\frac{2}{r}\right)\Vert p_x*uvu_x\Vert_{L^r} + \Vert p_x*uv_{xx}u_x\Vert_{L^r}\le \left(\left(1+\frac{2}{r}\right)\Vert v\Vert_{L^{\infty}} + \Vert v_{xx}\Vert_{L^{\infty}}\right)\Vert p_x\Vert_{L^{\frac{r}{r-1}}}\Vert u\Vert_{L^{r}}\Vert u_x\Vert_{L^r}.
	\end{align*}
	According to Proposition \ref{Proposition 2.7}, we have
	\begin{align*}
		\frac{\mathrm{d}\Vert u_x\Vert_{L^r}^r}{\mathrm{d}t} &\le r\Vert u_{xt}\Vert_{L^r}\Vert u_x\Vert_{L^r}^{r-1} \\&\le \left(\frac{r}{2}\right)^{\frac{1}{r}}\left(r+1+7\left(r-1\right)^{\frac{r-1}{r}}\right)\Vert n\Vert_{L^{\infty}}\Vert u\Vert_{W^{1,r}}^{r+1}.
	\end{align*}
	Integrating the above inequality with respect to $t$, we derive that
	\begin{align}\label{(5.5)}
		\Vert u_x\Vert_{L^r}^r \le \Vert u_{0,x}\Vert_{L^r}^r+\left(\frac{r}{2}\right)^{\frac{1}{r}}\left(r+1+7\left(r-1\right)^{\frac{r-1}{r}}\right)\int_{0}^{t} \Vert n\left(s\right)\Vert_{L^{\infty}}\Vert u\left(s\right)\Vert_{W^{1,r}}^{r+1}ds.
	\end{align}
	Now using \eqref{(5.4)} and \eqref{(5.5)}, one gets that
	\begin{align}\label{(5.6)}
		\Vert u\Vert_{W^{1,r}}^r \le \Vert u_{0,x}\Vert_{L^r}^r+\left(\frac{r}{2}\right)^{\frac{1}{r}}\left(2r+1+14\left(r-1\right)^{\frac{r-1}{r}}\right)\int_{0}^{t} \Vert n\left(s\right)\Vert_{L^{\infty}}\Vert u\left(s\right)\Vert_{W^{1,r}}^{r+1}ds.
	\end{align}
	Together with \eqref{(5.3)} and \eqref{(5.6)} gives
	\begin{align*}
		\Vert n\Vert_{L^{\infty}}+\Vert u\Vert_{W^{1,r}}^r \le \Vert n_0\Vert_{L^{\infty}}+\Vert u_0\Vert_{W^{1,r}}&+\int_{0}^{t} \left(1+\frac{2}{r}\right)\left(\frac{r}{2}\right)^{\frac{1}{r}}\Vert n\left(s\right)\Vert^2_{L^{\infty}}\Vert u\left(s\right)\Vert_{W^{1,r}}\\&+\left(\frac{r}{2}\right)^{\frac{1}{r}}\left(2r+1+14\left(r-1\right)^{\frac{r-1}{r}}\right)\Vert n\left(s\right)\Vert_{L^{\infty}}\Vert u\left(s\right)\Vert_{W^{1,r}}^{r+1}ds. 
	\end{align*}
	Therefore, we have
	\begin{align*}
		\Vert n\Vert_{L^{\infty}}+\Vert u\Vert_{W^{1,r}}^r \le \Vert n_0\Vert_{L^{\infty}}+\Vert u_0\Vert_{W^{1,r}}^r+\left(\frac{r}{2}\right)^{\frac{1}{r}}\left(2r+2+\frac{2}{r}+14\left(r-1\right)^{\frac{r-1}{r}}\right)\int_{0}^{t} \left(\Vert n\Vert_{L^{\infty}}+\Vert u\Vert_{W^{1,r}}^{r}\right)^{\frac{2r+1}{r}}ds.
	\end{align*}
	Now we obtain
	\begin{align*}
		\Vert u\Vert_{W^{1,r}}^r+\Vert n\Vert_{L^{\infty}} \le 2\left(\Vert u_0\Vert_{W^{1,r}}^r+\Vert n_0\Vert_{L^{\infty}}\right),
	\end{align*}
	with
	\begin{align*}
		t \le T_4=\frac{1}{8\left(r^{\frac{1+r}{r}}+r^{\frac{1}{r}}+r^{\frac{1-r}{r}}+7r^{\frac{1}{r}}\left(r-1\right)^{\frac{r-1}{r}}\right)\left(\Vert u_0\Vert_{W^{1,r}}^r+\Vert n_0\Vert_{L^{\infty}}\right)^{\frac{r+1}{r}}}.
	\end{align*}
	The proof is thus complete.
\end{proof}

\begin{prop}\label{Proposition 5.2}
	Assume that $u_0 \in W^{1,r}\left(\mathbb{R}\right)$, $v_0 \in L^{\infty}\left(\mathbb{R}\right)$ and there exists a point $x_0$ such that $v_0\left(x_0\right) > 0.$ Let $T_3$ be the maximal existence time of the corresponding strong solution $\left(u,v\right)$ to system \eqref{(1.5)} with $b=1+\frac{2}{r}$. Then we have
	\begin{align*}
		v \ge \frac{v_0\left(x_0\right)}{2},
	\end{align*}
	with
	\begin{align*}
		t \le T_5=\frac{v_0\left(x_0\right)}{8\left(r^{\frac{1+r}{r}}+r^{\frac{1}{r}}+r^{\frac{1-r}{r}}+7r^{\frac{1}{r}}\left(r-1\right)^{\frac{r-1}{r}}\right)\left(\Vert u_0\Vert_{W^{1,r}}^r+\Vert n_0\Vert_{L^{\infty}}\right)^{\frac{2r+1}{r}}} \le T_4.
	\end{align*} 
\end{prop}
\begin{proof}
	Consider the system \eqref{(1.5)} along the characteristics $q\left(t,x\right)$ with $b=1+\frac{2}{r}$, we have
	\begin{align*}
		v_t+p*\left(\left(1+\frac{2}{r}\right)uvv_x+\left(2-\frac{2}{r}\right)uv_xv_{xx}+u_xv_x^2+vu_xv_{xx}\right)+p_x*vv_xu_x=0.
	\end{align*} 
	By the Young's inequality, one gets
	\begin{align*}
		&\Vert p*uv_xv\Vert_{L^{\infty}} \le \Vert v\Vert_{L^{\infty}}\Vert v_x\Vert_{L^{\infty}}\Vert p\Vert_{L^{\frac{r}{r-1}}}\Vert u\Vert_{L^{r}},
		\\&\Vert p*uv_xv\Vert_{L^{\infty}} + \Vert p*uv_{xx}v_x\Vert_{L^{\infty}} \le \left(\Vert v_x\Vert_{L^{\infty}}\Vert v_{xx}\Vert_{L^{\infty}} + \Vert v\Vert_{L^{\infty}}\Vert v_x\Vert_{L^{\infty}}\right)\Vert p\Vert_{L^{\frac{r}{r-1}}}\Vert u\Vert_{L^{r}},
		\\&\Vert p*u_xv_xv_x\Vert_{L^{\infty}} + \Vert p*u_xv_{xx}v\Vert_{L^{\infty}}  \le \left(\Vert v_x\Vert_{L^{\infty}}\Vert v_x\Vert_{L^{\infty}} + \Vert v_{xx}\Vert_{L^{\infty}}\Vert v\Vert_{L^{\infty}}\right)\Vert p\Vert_{L^{\frac{r}{r-1}}}\Vert u_x\Vert_{L^{r}},
		\\&\Vert p_x*u_xv_xv\Vert_{L^{\infty}} \le \Vert p_x\Vert_{L^{\frac{r}{r-1}}}\Vert v_x\Vert_{L^{\infty}}\Vert v\Vert_{L^{\infty}}\Vert u_x\Vert_{L^{r}}.
	\end{align*}
	Hence,
	\begin{align*}
		|v_t(t,q(t,x))| \le \left(4-\frac{1}{r}\right)\left(\frac{2r-2}{r}\right)^{\frac{r-1}{r}}\Vert n\left(t\right)\Vert^2_{L^{\infty}}\Vert u\left(t\right)\Vert_{W^{1,r}}.
	\end{align*}
	It is easy to check that integrating the above inequality with respect to $t$, we deduce that
	\begin{align*}
		&v \ge v_0 -\left(4-\frac{1}{r}\right)\left(\frac{2r-2}{r}\right)^{\frac{r-1}{r}}\int_{0}^{t} \Vert n\left(s\right)\Vert^2_{L^{\infty}}\Vert u\left(s\right)\Vert_{W^{1,1}}ds,
		\\&v \le v_0 + \left(4-\frac{1}{r}\right)\left(\frac{2r-2}{r}\right)^{\frac{r-1}{r}}\int_{0}^{t} \Vert n\left(s\right)\Vert^2_{L^{\infty}}\Vert u\left(s\right)\Vert_{W^{1,1}}ds. 
	\end{align*}
	Together with the Proposition \ref{Proposition 5.1} and the fact that $q(t,x)$ is diffeomorphism of $\mathbb{R}$, we obtain that
	\begin{align*}
		v(t,q(t,x_0)) \ge \frac{v_0\left(x_0\right)}{2},
	\end{align*}
	with
	\begin{align*}
		t \le T_5=\frac{v_0\left(x_0\right)}{8\left(r^{\frac{1+r}{r}}+r^{\frac{1}{r}}+r^{\frac{1-r}{r}}+7r^{\frac{1}{r}}\left(r-1\right)^{\frac{r-1}{r}}\right)\left(\Vert u_0\Vert_{W^{1,r}}^r+\Vert n_0\Vert_{L^{\infty}}\right)^{\frac{2r+1}{r}}} \le T_4,
	\end{align*}
	This completes the proof.
\end{proof}

Now, we are in a position to show the blow-up phenomenon of system \eqref{(1.5)}.\\
\textbf{The proof of Theorem \ref{Theorem 1.4}}:\\
 Differentiating the system \eqref{(1.5)} to $x$, we have
	\begin{align*}
		u_{xt}+\frac{1}{r}vu_x^2-uv_xu_x+vuu_{xx}&+p_x*\left(\left(1+\frac{2}{r}\right)vuu_x-uu_xv_{xx}-\left(1-\frac{1}{r}\right)v_xu_x^2\right)\\&+p*\left(2uv_xu_x+\left(1-\frac{1}{r}\right)vu_x^2\right)=0.
	\end{align*}
	Hence, we obtain that
\begin{align*}
	&\left( u_{xt} + \frac{1}{r}vu_x^2 - uv_xu_x + p_x*\left( \left(1+\frac{2}{r}\right)vuu_x - uu_xv_{xx} - \left(1-\frac{1}{r}\right)v_xu_x^2 \right) \right. \\
	&\qquad\qquad\qquad\qquad \left. + p*\left( 2uv_xu_x + \left(1-\frac{1}{r}\right)vu_x^2 \right) \right)(t, q(t,x)) = 0.
\end{align*}
	By the Young's inequality, one gets
	\begin{align*}
		&2\Vert p*uv_xu_x\Vert_{L^p} + \left(1-\frac{1}{r}\right)\Vert p*u_xvu_x\Vert_{L^p} \le \left(2\Vert u\Vert_{L^{r}}\Vert v_x\Vert_{L^{\infty}} + \left(1-\frac{1}{r}\right)\Vert u_x\Vert_{L^{r}}\Vert v\Vert_{L^{\infty}} \right)\Vert p\Vert_{L^{\frac{r}{r-1}}}\Vert u_x\Vert_{L^r},
		\\&\Vert p_x*u_xv_xu_x\Vert_{L^p} \le \Vert p_x\Vert_{L^{\frac{r}{r-1}}}\Vert u_x\Vert_{L^{r}}\Vert v_x\Vert_{L^{\infty}}\Vert u_x\Vert_{L^r},
		\\&\left(1+\frac{2}{r}\right)\Vert p_x*uvu_x\Vert_{L^p} + \Vert p_x*uv_{xx}u_x\Vert_{L^p} \le \left(\left(1+\frac{2}{r}\right)\Vert v\Vert_{L^{\infty}} + \Vert v_{xx}\Vert_{L^{\infty}}\right)\Vert p_x\Vert_{L^{\frac{r}{r-1}}}\Vert u\Vert_{L^{r}}\Vert u_x\Vert_{L^r}.
	\end{align*}
	Together with the Proposition \ref{Proposition 2.7}, Proposition \ref{Proposition 5.1} and Proposition \ref{Proposition 5.2}, we derive that
	\begin{align*}
		\frac{\mathrm{d}u_x\left(t,q\left(t,x_0\right)\right)}{\mathrm{d}t}  &\le -\frac{v_0\left(x_0\right)}{4r}u_x^2+\frac{1}{v_0\left(x_0\right)}\Vert n\Vert_{L^{\infty}}^2\Vert u\Vert^2_{L^{\infty}}+7\left(2\right)^{-\frac{1}{r}}\left(\frac{r-1}{r}\right)^{\frac{r-1}{r}}\Vert n\Vert_{L^{\infty}}\Vert u\Vert^2_{L^{\infty}}
		\\ &\le -\frac{v_0\left(x_0\right)}{4r}u_x^2+\frac{1}{v_0\left(x_0\right)}\left(\frac{r}{2}\right)^{\frac{2}{r}}\Vert n\Vert_{L^{\infty}}^2\Vert u\Vert^2_{W^{1,r}}+7\left(2\right)^{-\frac{1}{r}}\left(\frac{r-1}{r}\right)^{\frac{r-1}{r}}\Vert n\Vert_{L^{\infty}}\Vert u\Vert^2_{L^{\infty}}
		\\&\le -a_2f_2+c_2,
	\end{align*}
	where
	\begin{align*}
		&f_2:=u_x\left(t,q\left(t,x_0\right)\right),
		\\&a_2:=\frac{v_0(x_0)}{4r},
	\end{align*}
	and
	\begin{align*}
		c_2:=\frac{4}{v_0\left(x_0\right)}r^{\frac{2}{r}}\left(\Vert n_0\Vert_{L^{\infty}}+\Vert u_0\Vert_{W^{1,r}}^r\right)^{\frac{2r+2}{r}}+14\left(\frac{r^2}{2}\right)^{\frac{1}{r}}\left(\frac{r-1}{r}\right)^{\frac{r-1}{r}}\left(\Vert n_0\Vert_{L^{\infty}}+\Vert u_0\Vert_{W^{1,r}}^r\right)^{\frac{r+2}{r}}.
	\end{align*} 
	Thanks to \eqref{(1.6)}, we then get
	\begin{align*}
		\frac{1}{\sqrt{c_2v_0\left(x_0\right)}}\ln\left(\frac{\sqrt{v_0\left(x_0\right)}f_2\left(0\right)-\sqrt{c_2}}{\sqrt{v_0\left(x_0\right)}f_2\left(0\right)+\sqrt{c_2}}\right) \le T_5.
	\end{align*}
	Applying Lemma \ref{Lemma 2.8}, we have
	\begin{align*}
		\lim_{t\rightarrow T_3}f_2(t)=-\infty,
	\end{align*}
	with
	\begin{align*}
		T_3 \le \frac{1}{\sqrt{c_2v_0\left(x_0\right)}}\ln\left(\frac{\sqrt{v_0\left(x_0\right)}f_2\left(0\right)-\sqrt{c_2}}{\sqrt{v_0\left(x_0\right)}f_2\left(0\right)+\sqrt{c_2}}\right) \le T_5,
	\end{align*}
	which along with Lemma \ref{Lemma 2.8} yields the desired result.
	$\hfill\square$

\begin{rema}
	For the variable $v$, we can also get a similar result.
\end{rema}

\noindent\textbf{Conflict of interest}: This work does not have any conflicts of interest.\\
\noindent {\bf Acknowledgments.} Z. Yin was supported by the National Natural Science Foundation of China (No.12571261).



\phantomsection
\addcontentsline{toc}{section}{\refname}
\bibliographystyle{abbrv} 
\bibliography{CMref}

@article{holm1998euler,
 AUTHOR = {Holm, Darryl D. and Marsden, Jerrold E. and Ratiu, Tudor S.},
     TITLE = {The {E}uler-{P}oincar\'e{} equations and semidirect products
              with applications to continuum theories},
   JOURNAL = {Adv. Math.},
  FJOURNAL = {Advances in Mathematics},
    VOLUME = {137},
      YEAR = {1998},
    NUMBER = {1},
     PAGES = {1--81},
      ISSN = {0001-8708,1090-2082},
   MRCLASS = {58F05 (58E30 58F40 70H30 73B99 73V25)},
  MRNUMBER = {1627802},
MRREVIEWER = {Jair\ Koiller},
       DOI = {10.1006/aima.1998.1721},
       URL = {https://doi.org/10.1006/aima.1998.1721},
}

@article{bressan2007global,
  AUTHOR = {Bressan, Alberto and Constantin, Adrian},
     TITLE = {Global conservative solutions of the {C}amassa-{H}olm
              equation},
   JOURNAL = {Arch. Ration. Mech. Anal.},
  FJOURNAL = {Archive for Rational Mechanics and Analysis},
    VOLUME = {183},
      YEAR = {2007},
    NUMBER = {2},
     PAGES = {215--239},
      ISSN = {0003-9527,1432-0673},
   MRCLASS = {35Q53 (35L45)},
  MRNUMBER = {2278406},
MRREVIEWER = {Enzo\ Vitillaro},
       DOI = {10.1007/s00205-006-0010-z},
       URL = {https://doi.org/10.1007/s00205-006-0010-z},
}

@article{chen2016blowup,
  AUTHOR = {Chen, Robin Ming and Guo, Fei and Liu, Yue and Qu, Changzheng},
     TITLE = {Analysis on the blow-up of solutions to a class of integrable
              peakon equations},
   JOURNAL = {J. Funct. Anal.},
  FJOURNAL = {Journal of Functional Analysis},
    VOLUME = {270},
      YEAR = {2016},
    NUMBER = {6},
     PAGES = {2343--2374},
      ISSN = {0022-1236,1096-0783},
   MRCLASS = {35Q53 (35B44 35G25)},
  MRNUMBER = {3460243},
MRREVIEWER = {Corentin\ Audiard},
       DOI = {10.1016/j.jfa.2016.01.017},
       URL = {https://doi.org/10.1016/j.jfa.2016.01.017},
}

@article{constantin2010inverse, 
AUTHOR = {Constantin, Adrian and Ivanov, Rossen I. and Lenells, Jonatan},
     TITLE = {Inverse scattering transform for the {D}egasperis-{P}rocesi
              equation},
   JOURNAL = {Nonlinearity},
  FJOURNAL = {Nonlinearity},
    VOLUME = {23},
      YEAR = {2010},
    NUMBER = {10},
     PAGES = {2559--2575},
      ISSN = {0951-7715,1361-6544},
   MRCLASS = {37K15 (34A55 34L25 35P25 35Q53 35R30 37K20)},
  MRNUMBER = {2683782},
MRREVIEWER = {Dmitry\ G.\ Shepelsky},
       DOI = {10.1088/0951-7715/23/10/012},
       URL = {https://doi.org/10.1088/0951-7715/23/10/012},
}

@article{constantin2009hydrodynamical,
 AUTHOR = {Constantin, Adrian and Lannes, David},
     TITLE = {The hydrodynamical relevance of the {C}amassa-{H}olm and
              {D}egasperis-{P}rocesi equations},
   JOURNAL = {Arch. Ration. Mech. Anal.},
  FJOURNAL = {Archive for Rational Mechanics and Analysis},
    VOLUME = {192},
      YEAR = {2009},
    NUMBER = {1},
     PAGES = {165--186},
      ISSN = {0003-9527,1432-0673},
   MRCLASS = {35Q53 (35R35 74J30 76A02 76B15)},
  MRNUMBER = {2481064},
MRREVIEWER = {Bengt\ O.\ Enflo},
       DOI = {10.1007/s00205-008-0128-2},
       URL = {https://doi.org/10.1007/s00205-008-0128-2},
}

@article{degasperis2002new,
  AUTHOR = {Degasperis, A. and Kholm, D. D. and Khon, A. N. I.},
     TITLE = {A new integrable equation with peakon solutions},
   JOURNAL = {Teoret. Mat. Fiz.},
  FJOURNAL = {Teoreticheskaya i Matematicheskaya Fizika},
    VOLUME = {133},
      YEAR = {2002},
    NUMBER = {2},
     pages = {170--183},
      ISSN = {0564-6162,2305-3135},
   MRCLASS = {37K10 (35Q53)},
  MRNUMBER = {2001531},
       DOI = {10.1023/A:1021186408422},
       URL = {https://doi.org/10.1023/A:1021186408422},
}

@article{degasperis1999asymptotic,
  title={Asymptotic integrability},
  author={Degasperis, Antonio and Procesi, Michela and others},
  journal={Symmetry and perturbation theory},
  volume={1},
  number={1},
  pages={23--37},
  year={1999},
  publisher={Rome}
}

@article{himonas2012cauchy,
   AUTHOR = {Himonas, A. Alexandrou and Holliman, Curtis},
     TITLE = {The {C}auchy problem for the {N}ovikov equation},
   JOURNAL = {Nonlinearity},
  FJOURNAL = {Nonlinearity},
    VOLUME = {25},
      YEAR = {2012},
    NUMBER = {2},
     PAGES = {449--479},
      ISSN = {0951-7715,1361-6544},
   MRCLASS = {35Q53 (35A35 35B30)},
  MRNUMBER = {2876876},
MRREVIEWER = {Luiz\ Gustavo\ Farah},
       DOI = {10.1088/0951-7715/25/2/449},
       URL = {https://doi.org/10.1088/0951-7715/25/2/449},
}

@article{himonas2013holder,
   AUTHOR = {Himonas, A. Alexandrou and Holmes, John},
     TITLE = {H\"older continuity of the solution map for the {N}ovikov
              equation},
   JOURNAL = {J. Math. Phys.},
  FJOURNAL = {Journal of Mathematical Physics},
    VOLUME = {54},
      YEAR = {2013},
    NUMBER = {6},
     PAGES = {061501, 11},
      ISSN = {0022-2488,1089-7658},
   MRCLASS = {35Q53 (35B65)},
  MRNUMBER = {3112520},
MRREVIEWER = {Alessandro\ Arsie},
       DOI = {10.1063/1.4807729},
       URL = {https://doi.org/10.1063/1.4807729},
}

@article{coclite2006well,
   AUTHOR = {Coclite, Giuseppe M. and Karlsen, Kenneth H.},
     TITLE = {On the well-posedness of the {D}egasperis-{P}rocesi equation},
   JOURNAL = {J. Funct. Anal.},
  FJOURNAL = {Journal of Functional Analysis},
    VOLUME = {233},
      YEAR = {2006},
    NUMBER = {1},
     PAGES = {60--91},
      ISSN = {0022-1236,1096-0783},
   MRCLASS = {35Q53 (35B30 35D05 35L65 37K10)},
  MRNUMBER = {2204675},
MRREVIEWER = {John\ Albert},
       DOI = {10.1016/j.jfa.2005.07.008},
       URL = {https://doi.org/10.1016/j.jfa.2005.07.008},
}

@article{escher2009initial,
   AUTHOR = {Escher, Joachim and Yin, Zhaoyang},
     TITLE = {Initial boundary value problems for nonlinear dispersive wave
              equations},
   JOURNAL = {J. Funct. Anal.},
  FJOURNAL = {Journal of Functional Analysis},
    VOLUME = {256},
      YEAR = {2009},
    NUMBER = {2},
     PAGES = {479--508},
      ISSN = {0022-1236,1096-0783},
   MRCLASS = {35Q53},
  MRNUMBER = {2476950},
MRREVIEWER = {Sebastian\ Herr},
       DOI = {10.1016/j.jfa.2008.07.010},
       URL = {https://doi.org/10.1016/j.jfa.2008.07.010},
}

@article{novikov2009generalizations,
   AUTHOR = {Novikov, Vladimir},
     TITLE = {Generalizations of the {C}amassa-{H}olm equation},
   JOURNAL = {J. Phys. A},
  FJOURNAL = {Journal of Physics. A. Mathematical and Theoretical},
    VOLUME = {42},
      YEAR = {2009},
    NUMBER = {34},
     PAGES = {342002, 14},
      ISSN = {1751-8113,1751-8121},
   MRCLASS = {35Q53},
  MRNUMBER = {2530232},
       DOI = {10.1088/1751-8113/42/34/342002},
       URL = {https://doi.org/10.1088/1751-8113/42/34/342002},
}

@article{geng2009extension,
  AUTHOR = {Geng, Xianguo and Xue, Bo},
     TITLE = {An extension of integrable peakon equations with cubic
              nonlinearity},
   JOURNAL = {Nonlinearity},
  FJOURNAL = {Nonlinearity},
    VOLUME = {22},
      YEAR = {2009},
    NUMBER = {8},
     PAGES = {1847--1856},
      ISSN = {0951-7715,1361-6544},
   MRCLASS = {37K10 (35Q51 35Q53 37K05)},
  MRNUMBER = {2525813},
MRREVIEWER = {Yoshimasa\ Matsuno},
       DOI = {10.1088/0951-7715/22/8/004},
       URL = {https://doi.org/10.1088/0951-7715/22/8/004},
}

@article{hou2013algebro,
  AUTHOR = {Hou, Yu and Zhao, Peng and Fan, Engui and Qiao, Zhijun},
     TITLE = {Algebro-geometric solutions for the {D}egasperis-{P}rocesi
              hierarchy},
   JOURNAL = {SIAM J. Math. Anal.},
  FJOURNAL = {SIAM Journal on Mathematical Analysis},
    VOLUME = {45},
      YEAR = {2013},
    NUMBER = {3},
     PAGES = {1216--1266},
      ISSN = {0036-1410,1095-7154},
   MRCLASS = {35Q53},
  MRNUMBER = {3049654},
MRREVIEWER = {Ademir\ Fernando\ Pazoto},
       DOI = {10.1137/12089689X},
       URL = {https://doi.org/10.1137/12089689X},
}

@article{qiao2004integrable,
  AUTHOR = {Qiao, Zhijun},
     TITLE = {Integrable hierarchy, {$3\times3$} constrained systems, and
              parametric solutions},
   JOURNAL = {Acta Appl. Math.},
  FJOURNAL = {Acta Applicandae Mathematicae},
    VOLUME = {83},
      YEAR = {2004},
    NUMBER = {3},
     PAGES = {199--220},
      ISSN = {0167-8019,1572-9036},
   MRCLASS = {37K10 (35Q35 35Q53)},
  MRNUMBER = {2092400},
MRREVIEWER = {Vladislav\ G.\ Dubrovsky},
       DOI = {10.1023/B:ACAP.0000038872.88367.dd},
       URL = {https://doi.org/10.1023/B:ACAP.0000038872.88367.dd},
}

@article{yan2018new,
   AUTHOR = {Yan, Kai and Qiao, Zhijun and Zhang, Yufeng},
     TITLE = {On a new two-component {$b$}-family peakon system with cubic
              nonlinearity},
   JOURNAL = {Discrete Contin. Dyn. Syst.},
  FJOURNAL = {Discrete and Continuous Dynamical Systems},
    VOLUME = {38},
      YEAR = {2018},
    NUMBER = {11},
     PAGES = {5415--5442},
      ISSN = {1078-0947,1553-5231},
   MRCLASS = {35F55 (35C08 35G25)},
  MRNUMBER = {3917775},
       DOI = {10.3934/dcds.2018239},
       URL = {https://doi.org/10.3934/dcds.2018239},
}

@article{yin2003cauchy,
  AUTHOR = {Yin, Zhaoyang},
     TITLE = {On the {C}auchy problem for an integrable equation with peakon
              solutions},
   JOURNAL = {Illinois J. Math.},
  FJOURNAL = {Illinois Journal of Mathematics},
    VOLUME = {47},
      YEAR = {2003},
    NUMBER = {3},
     PAGES = {649--666},
      ISSN = {0019-2082,1945-6581},
   MRCLASS = {35Q58 (35D05 35G25 37K40)},
  MRNUMBER = {2007229},
MRREVIEWER = {Maciej\ B\l aszak},
       URL = {http://projecteuclid.org/euclid.ijm/1258138186},
}

@article{yin2004global,
 AUTHOR = {Yin, Zhaoyang},
     TITLE = {Global weak solutions for a new periodic integrable equation
              with peakon solutions},
   JOURNAL = {J. Funct. Anal.},
  FJOURNAL = {Journal of Functional Analysis},
    VOLUME = {212},
      YEAR = {2004},
    NUMBER = {1},
     PAGES = {182--194},
      ISSN = {0022-1236,1096-0783},
   MRCLASS = {35Q53 (35B10 35D05 37K10)},
  MRNUMBER = {2065241},
MRREVIEWER = {Pavel\ Krej\v c\'i},
       DOI = {10.1016/j.jfa.2003.07.010},
       URL = {https://doi.org/10.1016/j.jfa.2003.07.010},
}

@article{camassa1993integrable,
  AUTHOR = {Camassa, Roberto and Holm, Darryl D.},
     TITLE = {An integrable shallow water equation with peaked solitons},
   JOURNAL = {Phys. Rev. Lett.},
  FJOURNAL = {Physical Review Letters},
    VOLUME = {71},
      YEAR = {1993},
    NUMBER = {11},
     PAGES = {1661--1664},
      ISSN = {0031-9007,1079-7114},
   MRCLASS = {35Q51 (58F07 76B15 76B25)},
  MRNUMBER = {1234453},
       DOI = {10.1103/PhysRevLett.71.1661},
       URL = {https://doi.org/10.1103/PhysRevLett.71.1661},
}

@article{camassa1994new,
  AUTHOR  = {Camassa, R. and Holm, D. D. and Hyman, J. M.},
     TITLE  = {A new integrable shallow water equation},
   JOURNAL = {Adv. Appl. Mech.},
    VOLUME  = {31},
      YEAR  = {1994},
     PAGES  = {1--33},
       DOI  = {10.1016/S0065-2156(08)70254-0},
}

@article{danchin2001remarks,
 AUTHOR = {Danchin, Rapha\"el},
     TITLE = {A few remarks on the {C}amassa-{H}olm equation},
   JOURNAL = {Differential Integral Equations},
  FJOURNAL = {Differential and Integral Equations. An International Journal
              for Theory \& Applications},
    VOLUME = {14},
      YEAR = {2001},
    NUMBER = {8},
     PAGES = {953--988},
      ISSN = {0893-4983},
   MRCLASS = {35Q53 (35B30 35B40 35G25 76B03 76B15)},
  MRNUMBER = {1827098},
MRREVIEWER = {Kenji\ Nakanishi},
}

@article{himonas2016novikov,
  AUTHOR = {Himonas, A. Alexandrou and Mantzavinos, Dionyssios},
     TITLE = {The initial value problem for a {N}ovikov system},
   JOURNAL = {J. Math. Phys.},
  FJOURNAL = {Journal of Mathematical Physics},
    VOLUME = {57},
      YEAR = {2016},
    NUMBER = {7},
     PAGES = {071503, 21},
      ISSN = {0022-2488,1089-7658},
   MRCLASS = {35Q53 (35B30 35C07)},
  MRNUMBER = {3529006},
MRREVIEWER = {Ademir\ Pastor},
       DOI = {10.1063/1.4959774},
       URL = {https://doi.org/10.1063/1.4959774},
}

@book{BCD,
  author    = {Bahouri, H. and Chemin, J.-Y. and Danchin, R.},
  title     = {Fourier Analysis and Nonlinear Partial Differential Equations},
  publisher = {Springer},
  address   = {Heidelberg},
  year      = {2011},
  series    = {Grundlehren der Mathematischen Wissenschaften},
  volume    = {343},
  isbn      = {978-3-642-16830-7},
  doi       = {10.1007/978-3-642-16831-4}
}

@article{Li2025,
  AUTHOR  = {Li, J. L. and Yu, Y. H. and Zhu, W. P.},
  TITLE   = {Ill-posedness for the periodic {C}amassa-{H}olm type equations in critical Besov spaces},
  JOURNAL = {Ann. Mat. Pura Appl. (1923-)},
  YEAR    = {2025},
  DOI     = {10.1007/s10231-025-01638-0},
}

@article {qiao2019persistence,
    AUTHOR = {Chen, Rong and Qiao, Zhijun and Zhou, Shouming},
     TITLE = {Persistence properties and wave-breaking criteria for the
              {G}eng-{X}ue system},
   JOURNAL = {Math. Methods Appl. Sci.},
  FJOURNAL = {Mathematical Methods in the Applied Sciences},
    VOLUME = {42},
      YEAR = {2019},
    NUMBER = {18},
     PAGES = {6999--7010},
      ISSN = {0170-4214,1099-1476},
   MRCLASS = {35Q53 (35A01 35B44 35D35)},
  MRNUMBER = {4037948},
       DOI = {10.1002/mma.5805},
       URL = {https://doi.org/10.1002/mma.5805},
}

\end{document}